\documentclass[a4paper,11pt]{amsart}

\usepackage{amsfonts,amsmath,mathrsfs,amssymb,amsthm,amscd,latexsym,amstext,amsxtra}
\usepackage[usenames,dvipsnames]{color}

\usepackage[left=2cm,right=2cm, top=3cm, bottom=3cm]{geometry}

\usepackage[all,cmtip]{xy}
\usepackage{enumerate}
\xyoption{all}
\usepackage{srcltx}
\usepackage{textcomp}
\usepackage{parskip}

\usepackage{amsthm, amsmath, amssymb,latexsym} 
\usepackage{amsfonts,epsfig,amscd}
\usepackage[]{fontenc}
\usepackage{xy}
\usepackage{enumerate}
\usepackage[]{fontenc}
\usepackage[all]{xy}

\newtheorem{theorem}{Theorem}[section]
\newtheorem{lemma}[theorem]{Lemma}
\newtheorem{corollary}[theorem]{Corollary}
\newtheorem{proposition}[theorem]{Proposition}
\newtheorem{remark}[theorem]{Remark}
\newtheorem{definition}[theorem]{Definition}

\newcommand{\nc}{\newcommand}
\nc{\cH}{{\mathcal H}}
\nc{\cA}{{\mathcal A}}
\nc{\cG}{{\mathcal G}}
\nc{\cC}{{\mathcal C}}
\nc{\cD}{{\mathcal D}}
\nc{\cO}{{\mathcal O}}
\nc{\cI}{{\mathcal I}}
\nc{\cB}{{\mathcal B}}
\nc{\cY}{{\mathcal Y}}
\nc{\cK}{{\mathcal K}}
\nc{\cX}{{\mathcal X}}
\nc{\cS}{{\mathcal S}}
\nc{\cE}{{\mathcal E}}
\nc{\cF}{{\mathcal F}}
\nc{\cZ}{{\mathcal Z}}
\nc{\cQ}{{\mathcal Q}}
\nc{\cN}{{\mathcal N}}
\nc{\cP}{{\mathcal P}}
\nc{\cL}{{\mathcal L}}
\nc{\cM}{{\mathcal M}}
\nc{\cT}{{\mathcal T}}
\nc{\cW}{{\mathcal W}}
\nc{\cU}{{\mathcal U}}
\nc{\cJ}{{\mathcal J}}
\nc{\cV}{{\mathcal V}}
\nc{\bH}{{\mathbb H}}
\nc{\bA}{{\mathbb A}}
\nc{\bG}{{\mathbb G}}
\nc{\bC}{{\mathbb C}}
\nc{\bO}{{\mathbb O}}
\nc{\bI}{{\mathbb I}}
\nc{\bB}{{\mathbb B}}
\nc{\bY}{{\mathbb Y}}
\nc{\bK}{{\mathbb K}}
\nc{\bX}{{\mathbb X}}
\nc{\bS}{{\mathbb S}}
\nc{\bE}{{\mathbb E}}
\nc{\bF}{{\mathbb F}}
\nc{\bZ}{{\mathbb Z}}
\nc{\bQ}{{\mathbb Q}}
\nc{\bN}{{\mathbb N}}
\nc{\bP}{{\mathbb P}}
\nc{\bL}{{\mathbb L}}
\nc{\bM}{{\mathbb M}}
\nc{\bT}{{\mathbb T}}
\nc{\bW}{{\mathbb W}}
\nc{\bU}{{\mathbb U}}
\nc{\bD}{{\mathbb D}}
\nc{\bJ}{{\mathbb J}}
\nc{\bV}{{\mathbb V}}
\nc{\bbZ}{{\mathbb Z}}
\nc{\bR}{{\mathbb R}}
\nc{\fm}{{\mathfrak m}}
\nc{\longra}{{\longrightarrow}}

\nc{\fr}{{\rightarrow}}
\nc{\co}{{\nabla}}

\nc{\cu}{{\overlineline{\nabla}}}
\nc{\gmc}{\nabla}
\nc{\mtin}[1]{\mbox{{\tiny #1}}}
\nc{\rankp}[1]{r_{\mbox{{\tiny #1}}}}
\nc{\pder}[1]{\frac{\partial }{\partial #1}}

\DeclareMathOperator{\Ima}{Im}

\DeclareMathOperator{\rk}{rk}
\DeclareMathOperator{\im}{im}

\DeclareMathOperator{\rank}{rank}

\DeclareMathOperator{\id}{id}

\DeclareMathOperator{\Cliff}{Cliff}

\title{Families of curves with Higgs field of arbitrarily large kernel.} 																											
\makeatletter

\@addtoreset{equation}{section}
\makeatother
\author{V\'ictor Gonz\'alez-Alonso}

\author{Sara Torelli}

\thanks{S. Torelli was supported by PRIN 2015 Moduli spaces and Lie Theory, INdAM - GNSAGA, FAR 2016 (PV)  Variet\`a algebriche, calcolo algebrico, grafi orientati e topologici and a Riemann Fellowship (Leibniz Universit\"at Hannover).}



\date{\today}

\begin{document}	
	
\maketitle
	
\bigskip
\begin{abstract}
In this note we consider the flat bundle $\cU$ and the kernel $\cK$ of the Higgs field naturally associated to any (polarized) variation of Hodge structures of weight 1. We study how strict the inclusion $\cU\subseteq\cK$ can be in the geometric case. More precisely, for any smooth projective curve $C$ of genus $g\geq 2$ and any $r=0,\ldots,g-1$, we construct non-isotrivial deformations of $C$ over a quasi-projecive base such that $\rk\cK=r$ and $\rk\cU\leq\frac{g+1}{2}$.
\end{abstract}

\section{Introduction and notations}

The Hodge bundle $\cH^{1,0}=f_\ast\omega_f$ of a one-parametric semistable family $f\colon S\rightarrow B$ of complex projective curves of genus $g$ (or more genenerally, of a polarized variation of Hodge structures of weight one) carries two natural vector subbundles: the flat unitary summand $\cU$ (cf. Fujita and Koll\'ar decompositions \cite{Fuj78b,Kol86}) and the kernel $\cK$ of the associated Higgs field (see Section 3 for more details). By definition there is an inculsion $\cU\subseteq\cK$, which must be an equality if $\cK=\cH^{1,0}$. Besides this trivial case, it is not difficult to explicitly construct (non-geometric) variations of Hodge structure over a disk where both $\rk\cU$ and $\rk\cK$ can be chosen arbitrarily (satisfying $\rk\cU\leq\rk\cK<g$). However, it is not clear whether this construction can provide geometric variations of Hodge structure, i.e. arising from a semistable family of curves, or on the contrary such geometric variations have some restrictions on the ranks of $\cU$ and $\cK$. In particular it is not clear when the equality $\cU=\cK$ holds in the geometric case.

The main result of this note is that $\cK$ can have any rank (between $0$ and $g-1$) also in geometric cases, with families containing an arbitrarily chosen curve, and even over (quasi-)projective base. If moreover the chosen curve has simple jacobian variety, the family can be chosen with $\cU=0$. More precisely, we prove:

\begin{theorem} \label{thm:main}
Let $C$ be any smooth projective curve of genus $g\geq2$. Then for any $0\leq r<g$ there is $f\colon\mathcal{C}\rightarrow B$, a non-isotrivial semistable one-dimensional family of deformations of $C$ over a projective base $B$, such that $\rk\mathcal{K}=r$ and $\rk\mathcal{U}\leq\frac{g+1}{2}$.
\end{theorem}

\begin{corollary} \label{cor:main}
If $C$ is a smooth projective curve of genus $g\geq2$ with simple jacobian variety, then for $0<r<g$ there is a deformation as in Theorem \ref{thm:main} with $\mathcal{U}=0$, hence $\mathcal{U}\subsetneq\mathcal{K}$.
\end{corollary}

Our motivation to study this question stems from the classification of fibred (irregular) surfaces. Indeed, in the recent work \cite{GST17} an upper bound for the rank of $\cU$ is obtained, depending on geometric invariants of the fibres like their genus and the general Clifford index, generalizing a previous result of \cite{BGN_Xiao_2015} on the relative irregularity. A closer look at the proof of that result shows that in some cases the upper bound is actually a bound for the rank of $\cK$. Therefore a better understanding of the inclusion $\cU\subseteq\cK$ could lead to improvements of the main result in \cite{GST17}. We notice that any $\cK$ bigger than $\cU$ is seminegative by a result of \cite{Z00}, highlighting how wild the behaviour of the Kodaira-Spencer map can be also in cases where the local Torelli theorem holds, and therefore adding importance to the study of those ranks as new numerical invariants.

A second possible application is the so called Coleman-Oort conjecture: roughly speaking, for high enough genus, the Torelli locus in $\cA_g$ should not contain positive-dimensional Shimura subvarieties. In \cite{CLZ16}, Chen, Lu and Zuo proved that, if the variation of Hodge structure associated to a Shimura curve $X\subset\cA_g$ has flat unitary bundle $\cU$ of $\rk\cU\geq\frac{4g+2}{5}$, then $X$ is not generically contained in the Torelli locus (i.e. $X$ intersects the Torelli locus at most in isolated points). Recently in \cite{CLZ18} the same authors proved that the same holds if $\rk\cU\leq\frac{2g-22}{7}$. Therefore Shimura curves in the Torelli locus cannot have $\cU$ of too big or too small rank. Since both bundles $\cU$ and $\cK$ for a curve $X$ in $\cA_g$ reflect the local structure of $X\subset\cA_g$, there could be a similar statement with $\rk\cK$ instead of $\rk\cU$. The relation between $\cU$ and $\cK$ with Massey products has also recently been used by Ghigi, Pirola and the second author in \cite{GPT19} to prove that any Shimura subvariety generically contained in the Torelli locus can have dimension at most $\frac{7g-2}{3}$. Altogether this supports the idea that a better understanding of the inculsion $\cU\subseteq\cK$ might lead to new insights for the Coleman-Oort conjecture.

Let us devote a few words to our techniques. Our main tool to estimate the ranks of $\cU$ and $\cK$ is Lemma \ref{lem-relrigRanks}, which leads us to focus on families that are supported on relatively rigid divisors (see Definition \ref{df:supp-div}). Roughly speaking, on a general fibre the first order infinitesimal deformation is described by a rigid divisor of the fibre, and these divisors glue along the family. However, supporting divisors are not canonically definable, not even the minimal ones. Indeed, any divisor of degree greater than $2g-2$ supports every deformation (e.g.$D=(2g-2)p$ for any point $p$), and thus any deformation has a minimal supporting divisor concentrated at any given point (with multiplicity). Nonetheless, for families obtained by deforming a branched finite covering, the theory developed by Horikawa in \cite{H73} allows to construct some natural minimal supporting divisors (see Lemma \ref{lem:Hori-divisor}).

At first sight one might expect that $\cU$ and $\cK$ coincide locally, and that a strict inequality $\cU\subsetneq\cK$ would be caused by monodromy if the base is not simply connected. But this is false, as the local nature of Lemma \ref{lem-relrigRanks} shows. This fact is strongly highlighted in Theorem \ref{thm:localfam} where some ad hoc local examples have been constructed.  We notice that the set of rigid divisors of a given degree of a curve is open and Zariski-dense in the Picard variety of the fixed degree, hence many families can be constructed in this way. 

The proof of Theorem \ref{thm:main} follows this line. We take a smooth projective curve $C$ of genus $g$ and for any $1\leq d\leq g$ we construct a ramified covering $C\rightarrow\mathbb{P}^1$ suitably ramified on a chosen rigid divisor $D\subset C$ of degree $d$. Then we consider a family of coverings obtained by moving $D$, which can be extended to a quasi-projective base. At this point the proof concludes as a straightforward application of Lemmas \ref{lem:Hori-divisor} and \ref{Lem-BoundKer}. 

The proof of Corollary \ref{cor:main} follows immediately by Theorem \ref{thm:main} because the monodromy of the flat bundle $\cU$ of those families is finite by a result of \cite{PT}. Thus a non vanishing $\cU$ would define a subvariety of the jacobian of a general fibre, contradicting its simplicity.

Although the constructions as given in the proof are already very explicit, in Section \ref{sec:cyclicexamples} we study in more detail some deformations of cyclic coverings inspired by the study on $\cU$ done in \cite{CatDet_Vector_2016,Lu18}. Our interest on these examples is motivated by the fact the corresponding $\cU$ has infinite monodromy, rank bigger than $(g+1) /2$ and moreover $\cK=\cU$, hence they look very different from our case where $\cU$ is smaller than $\cK$, has rank less than $(g+1) /2$ and finite monodromy. This kind of examples has been intensively studied with different approaches and objectives (see \cite{M10,Pet16,CFG15}). We notice that they are also interesting in our study since they admit a non vanishing flat bundle, which does not occur for a very general curve (see \cite[Theorem 3.13]{FGP}) and therefore we spend a few lines rephrasing some of their results using our tools.

The paper is organized as follows. In Section \ref{sec:horidiv} we relate the theories of supporting divisors and deformations of maps and prove Lemma \ref{lem:Hori-divisor}, which constructs a natural minimal supporting divisor by means of Horikawa's theory. In Section \ref{sec:rigidcase} we analyse the case of rigid supporting divisors (Lemma \ref{Lem-BoundKer}) and construct local families with any $\rk\cK$ (Theorem \ref{thm:localfam}). In Section \ref{sec:cyclicexamples} we consider in more detail deformations of cyclic coverings and compare them to those of \cite{CatDet_Vector_2016,Lu18}. 
Finally in section \ref{sec:pf:main} we prove Theorem \ref{thm:main} and Corollary \ref{cor:main}.

{\bf Acknowledgements:} We want to thank Gian Pietro Pirola, Lidia Stoppino, Xin Lu and Anand Deopurkar for some very fruitful discussions and enlightening ideas. Sara Torelli also thanks the Riemann Center and the Institute of Algebraic Geometry of Leibniz Universit\"at Hannover for their warm hospitality and support during her stay as Riemann Fellow which originated this work.

\section{Horikawa's deformation theory and supporting divisors}\label{sec:horidiv}
In this section we relate the theories of supporting divisors (see \cite{BGN_Xiao_2015}) and of deformation of maps (see \cite{H73}) to produce a somehow canonical supporting divisor for families of morphisms, which we use to estimate the ranks of $\mathcal{U}$ and $\mathcal{K}$. Let $f: \cC \rightarrow B$ be a smooth family of projective curves of genus $g\geq 2$ over a disk $B$.

\begin{definition}[Supporting divisors] \label{df:supp-div}
Let $C$ be a smooth projective curve and $\xi \in H^1(C,T_C)$ a first-order infinitesimal deformation of $C$.
An effective divisor $D$ in $C$ is a {\em supporting divisor} of $\xi$ if
\begin{equation}\label{eq:infinitesimal-supp-div}
\xi \in \ker\left(H^1\left(C,T_C\right) \longra H^1\left(C,T_C\left(D\right)\right)\right) = \im\left(H^0\left(D,T_C(D)_{|D}\right) \longra H^1\left(C,T_C\right)\right).
\end{equation}
A {\em minimal} supporting divisor is a supporting divisor $D$ with the extra property that any effective strict subdivisor $D' < D$ does not support $\xi$.

A {\em (minimal) supporting divisor} of a smooth family of curves $f: \cC \rightarrow B$ is an effective divisor $\cD \subset \cC$ such that on a general fibre $C_b = f^{-1}(b)$ the restriction $D_b = \cD_{|C_b}$ is a (minimal) supporting divisor of the infinitesimal defomation $\xi_b$ of $C_b$ induced by $f$.
\end{definition}

In the case of a family, up to shrinking $B$ we can always assume that a supporting divisor consists of sections of $f$ (possibly with coefficients).

For any divisor $D$ on a curve $C$ we denote by $r\left(D\right)=\dim\left|D\right|=h^0\left(C,\cO_C\left(D\right)\right)-1$ the dimension of its complete linear series, and by $\Cliff\left(D\right)=\deg D-2r\left(D\right)$ its Clifford index.

The following result is our basic tool to estimate the ranks of $\mathcal{U}$ and $\mathcal{K}$ in terms of invariants of a supporting divisor. 

\begin{lemma}[{\cite[Lemma 2.3 and Thm 2.4 in]{BGN_Xiao_2015} or \cite[Theorem 2.9]{GST17}}]\label{Lem-BoundKer}
Let $C$ be a projective curve of genus $g$, $\xi \in H^1(C,T_C)$ a first-order infinitesimal deformation and $\cup\xi: H^0(C,\omega_C) \rightarrow H^1(C,\mathcal{O}_C)$ the map induced by cup-product.
\begin{enumerate}
\item If $D$ is a divisor (in $C$) supporting $\xi$, then $H^0(C,\omega_C(-D))\subseteq \ker(\cup\xi)$ and hence
$$\dim \ker(\cup\xi) \geq g - (\deg D-r(D)).$$
\item If further $D$ supports $\xi$ {\em minimally}, then
$$\dim\ker(\cup\xi) \leq g - (\deg D - 2r(D)) = g - \Cliff(D).$$
\end{enumerate}
\end{lemma}

We notice that in particular, when a minimal supporting divisor $D$ is rigid, the estimates in Lemma \ref{Lem-BoundKer} lead to the equality $\dim \ker \cup \xi =g-\deg D$. 

This result is useful because $\mathcal{U} \subseteq \mathcal{K}$ and at a general $b \in B$ the Higgs field
$$\theta(b): f_*\omega_{\mathcal{C}/B} \otimes \mathbb{C}(b) \rightarrow R^1f_*\mathcal{O}_\cC \otimes \mathbb{C}(b) \otimes T_{B,b}^{\vee}$$
coincides with $\cup \xi_b$ (up to non-zero scalar, depending on the choice of an isomorphism $T_{B,b}^{\vee} \cong \mathbb{C}$).

In order to apply Lemma \ref{Lem-BoundKer} one has to construct a divisor minimally supporting $f$, but unfortunately such divisors are not unique and in general there is no canonical choice. In the case of  families of curves $f$ arising as deformations of morphisms onto a fixed curve, Horikawa's theory as developed in \cite{H73} gives a natural way to construct a supporting divisor using the so-called Horikawa characteristic class. 

We shortly recall the construction of the characteristic map and the relation to the Kodaira-Spencer class. Let $C'$ be a smooth projective curve. A family of morphisms of curves onto $C'$ is a morphism $(f,\Phi):\cC\to B\times C'$ such that $f:\cC\to B$ is a family of curves, and for any $b \in B$ the restriction $\pi_b=\Phi\circ i_b: C_b\to C'$ given by the inclusion $i_b:C_b = f^{-1}(b)\hookrightarrow\cC$ is a non constant morphism of curves. For any fixed $b\in B$, the morphism $\pi=\pi_b:C=C_b\to C'$ defines a short exact sequence 
\begin{equation}\label{ses-Hori}
\xymatrix@!R{
{0}& { T_{C}}& 			{\pi^\ast T_{C'}} & {{\mathcal{T}}_{\pi}} &  {0.}
\ar"1,1";"1,2"\ar"1,2";"1,3"^<<<<{d\pi}\ar"1,3";"1,4"^{p_{\pi}}\ar"1,4";"1,5"
}
\end{equation}
We can fix local coordinate systems $(\cU_i, (z_i,t))$ of $\cC$ and $(V_i,w_i)$ of $C'$ by choosing Stein open sets such that $\Phi(\cU_i)\subset V_i$ and where $t$ is the pull-back of a local coordinate  of $B$ around $b$. We denote by $\Phi_i$ the local expression of $\Phi$ with respect to these coordinate systems, i.e. $w_i=\Phi_i(z_i,t)$, and define a $0$-cochain of $ \pi^\ast T_{C'}$ by setting
$$s_i= \left(\frac{\partial \Phi_i}{\partial t}\right)_{|t=b}\frac{\partial}{\partial w_i}$$
on $U_i=\cU_i\cap C$. By applying $p_{\pi}$ we obtain a $0$-cochain of $\mathcal{T}_{\pi}$ given by $\tau_i=p_{\pi}\left(s_i\right)$ on $U_i$. These sections turn out to agree in the intersections $U_i \cap U_j$, giving rise to a section $\tau \in H^0(C, {\mathcal{T}}_{\pi})$ that is called {\em characteristic class} of $\pi$. The characteristic map 
$$\tau: T_bB \to H^0(C, {\mathcal{T}}_{\pi})$$
is the map that sends the generator $\frac{\partial}{\partial t}\in T_bB$ to the characteristic class $\tau \in H^0(C, {\mathcal{T}}_{\pi})$ defined as above.
By \cite[Proposition 1.4]{H73}, the characteristic map factors through the Kodaira-Spencer map $KS: T_0B \to H^1(C,T_C)$ as in 
\begin{equation}\label{dia-Hori}
\xymatrix@!R{
{T_0B}& 			{H^0(C,{\mathcal{T}}_{\pi})} & \\
&	{H^1(C,T_C)} & 
\ar"1,1";"2,2"_<<<{KS}\ar"1,1";"1,2"^{\tau}\ar"1,2";"2,2"^{\delta}
}
\end{equation}
where $\delta: H^0(C,\mathcal{T}_{\pi}) \rightarrow H^1(C,T_C)$ is the connecting homomorphism associated to (\ref{ses-Hori}). By construction this gives a one-to-one correspondence between the vector space $H^0(C,\cT_{\pi})$ and the set of equivalence classes of first-order deformations of the morphism $\pi: C \rightarrow C'$ (leaving $C'$ fixed).
	
The sheaf ${\mathcal{T}}_{\pi}$ can be more explictly described through the ramification divisor $R$ of $\pi$. Indeed, by definition of the ramification divisor there is an isomorphism $\pi^*T_{C'} \cong T_C(R)$ identifying $d\pi$ with the natural inclusion $T_C \hookrightarrow T_C(R)$. This in turn induces an isomorphism ${\mathcal{T}}_{\pi}\cong{T_C(R)}_{|R}$, which we use to construct a divisor minimally supporting $f$ in some cases.

In the previous setting, we say that the family $f$ is obtained from some $\pi=\pi_o: C \rightarrow C'$ by moving some (distinct) branch points $q_1,\ldots,q_k \in C'$ (while keeping the remaining branch points $q_{k+1},\ldots,q_n$ fixed) if there are some maps $\widetilde{q_1},\ldots,\widetilde{q_k}: B \rightarrow C'$ injective around $b=o$ and such that each $\pi_b: C_b \rightarrow C'$ is ramified over $\widetilde{q_1}(b),\ldots,\widetilde{q_k}(b),q_{k+1},\ldots,q_n$ with the same ramification type as for $b=o$.
	
\begin{lemma} \label{lem:Hori-divisor}
Keeping the above notations, suppose furthermore that for $i=1,\ldots,k$ there is only one ramification point $p_i$ over $q_i$, let $r_i+1$ be its ramification index and set $D=\sum_{i=1}^kr_ip_i$. If $H^0\left(C,T_C\left(D\right)\right)=0$, then any deformation of $\pi$ obtained moving $q_1,\ldots,q_k$ is minimally supported in $D$.
\end{lemma}
\begin{proof}
We consider the extension class $\xi \in H^1(C,T_C)$ induced by $f$ on $C=f^{-1}\left(o\right)$ and we prove that this is minimally supported over $D=\sum_{i=1}^k r_ip_i$. To do so, we compute $\xi$ by using the Horikawa's characteristic map. Fix first a local coordinate $t$ of $B$ centered in $o\in B$ and for each $i=1,\ldots,k$ choose local coordinates $z_i$ resp. $w_i$, centered on $p_i \in C$ resp. $q_i=\pi(p_i)\in C'$, such that $w_i=f(z_i,t)=z_i^{r_i+1}+t$. Then $\xi$ is given as
$$\xi = KS\left(\pder{t}\right)=\delta\left(\sum_{i=1}^k\pder{w_i}\right)=\delta\left(\sum_{i=1}^k\frac{1}{(r_i+1)z_i^{r_i}}\pder{z_i}\right).$$
Since $\sum_{i=1}^k\frac{1}{(r_i+1)z_i^{r_i}}\pder{z_i}$ is an element in $H^0(D, {T_C(D)}_{|D})\subset H^0(R,{T_C(R)_{|R}})$, this proves that
$$\xi\in \Ima (H^0(C,T_{C}(D)_{|D}) \longra H^1(C,T_{C}))$$
and so that $D$ supports $\xi$. We now prove that $D$ is minimally supporting $\xi$, i.e. that any effective subdivisor $D'<D$ does not support it. To do this it is enough to consider a subdivisor $D'=D-p_i$ of $D$ obtained by removing a point $p_i$ and then check this is not supporting $f$. We consider the short exact sequences $0\to T_C\to T_C(D')\to {T_C(D')}_{|D'}\to 0$ and $0\to T_C\to T_C(D)\to {T_C(D)}_{|D}\to 0$ induced by $D$ and $D'$ and we compare them through the inclusion $D'<D$. Since we have assumed $H^0(C, T_C(D))=0$, the map $H^0(C,T_{C}(D)_{|D}) \longra H^1(C,T_{C})$ is injective, hence it is enough to check that $\frac{1}{z_i^{r_i}}\pder{z_i}$ does not lie in $H^0(D',{T_C(D')}_{|D'})\subseteq H^0(D, {T_C(D)}_{|D})$. Indeed with the induced trivializations, the map $H^0(D', {T_C(D')}_{|D'}) \rightarrow H^0(D, {T_C(D)}_{|D})$ is given by multiplication with $z_i$, and sends the subset $<1,z_i, \dots , z_i^{r_i-2}>\otimes \{\frac{1}{z_i^{r_i-1}}\frac{\partial}{\partial z_i}\}$ of $H^0(D', {T_C(D')}_{|D'})$ to the subset $<z_i,z_i^2,\dots,z_i^{r_i-1}>\otimes \{\frac{1}{z_i^{r_i}}\frac{\partial}{\partial z_i}\} \subset H^0(D, {T_C(D)}_{|D})$, which obviously does not contain $\frac{1}{z_i^{r_i}}\frac{\partial}{\partial z_i}$.
\end{proof}

\begin{remark} \label{rmk:non-isotr}
Note that in the above setting, if $k\geq1$, there is a non-zero minimal supporting divisor. This implies that the family is not isotrivial, since the infinitesimal deformation is not zero.
\end{remark}

\section{The case of rigid supporting divisors}\label{sec:rigidcase}

In this section we study the ranks of the flat and kernel bundles for families supported on (relatively) rigid divisors and we also analyse the monodromy of the flat bundle. In particular, we construct families of curves with $\mathcal{K}$ of any given rank between $0$ and $g-1$. On the other hand, we show that $\rk \mathcal{U} \leq \frac{g+1}{2}$, hence in particular we can construct (local) families with $\mathcal{U} \subsetneq \mathcal{K}$. Note that $\rk\cK=g$ happens if and only if the family is isotrivial, and hence also $\cU=\cK$.

We start recalling the basic definitions around these bundles. Let $B$ be a complex curve and $f: \cC \rightarrow B$ be a non-isotrivial semistable family of projective curves of genus $g\geq 2$. Consider the Hodge bundle $f_* \omega_f$, where $\omega_f=\omega_\cC \otimes f^*\omega_B^{\vee}$. The Fujita decomposition \cite{Fuj78b} factors it as a direct sum $f_* \omega_f=\cU\oplus \cA$, with $\cU$ unitary flat and $\cA$ ample. If $\Gamma\subset B$ denotes the set of critical values (corresponding to singular fibers) and $\Upsilon=f^*\Gamma$, we can also consider the short exact sequence
$$0\to f^\ast \omega_B (\log \Gamma) \to \Omega^1_\cC( \log \Upsilon) \to \Omega^1_{\cC/B}(\log \Upsilon) \to 0.$$
Pushing it forward and using the canonical isomorphism $f_\ast\omega_f\simeq f_\ast\Omega^1_{\cC/B}( \log \Upsilon)$ we obtain a long exact sequence with connecting homomorphism
$$\theta: f_\ast \Omega^1_{\cC/B}\left(\log\Upsilon\right)\to R^1f_\ast f^\ast\omega_B(\log\Gamma) \simeq R^1f_\ast \cO_{\cC}\otimes \omega_B(\log \Gamma).$$
It is a morphism of vector bundles whose kernel $\cK=\ker \theta$ is a vector subbundle of $f_\ast\omega_f$.

\begin{definition}
We call the bundles $\cU$ and $\cK$ as defined above  the {\em flat bundle}  and {\em kernel bundle of $f$}, respectively.
\end{definition}

On the smooth locus of $f$ the objects as introduced above are naturally defined by the polarized variation of the Hodge structure. With a little abuse of notation, suppose for a moment that $f$ is smooth. In this case the Hodge bundle is
$$\cH^{1,0}=f_* \omega_f \subset \mathcal{H}^1=R^1f_*\mathbb{C} \otimes_{\mathbb{C}} \mathcal{O}_B,$$
where $\omega_f=\omega_\cC \otimes f^*\omega_B^{\vee} \cong \Omega_{\cC/B}^1$ because $f$ is smooth. The Gau\ss-Manin connection restricts to $\nabla_{\cH^{1,0}}:\cH^{1,0} \to \mathcal{\mathcal{H}}^1 \otimes \omega_B$, the {\em flat local system} over ${\cH^{1,0}}$ is $\mathbb{U}=\ker \nabla_{\cH^{1,0}}$ and the {\em flat subbundle} is $\mathcal{U}= \mathbb{U} \otimes_{\mathbb{C}} \mathcal{O}_B$. The Higgs field
$$\theta = p \circ \nabla_{\cH^{1,0}}: \cH^{1,0} \rightarrow \mathcal{H}^1 \rightarrow \mathcal{H}^1/\cH^{1,0} \otimes \omega_B \cong R^1f_*\mathcal{O}_\cC \otimes \omega_B$$
coincides with the connecting homomorphism
$$\theta: \cH^{1,0}=f_*\Omega_{\cC/B}^1 \rightarrow R^1f_*(f^*\omega_B) = R^1f_*\mathcal{O}_\cC \otimes \omega_B$$
arising by pushing forward the exact sequence $0 \rightarrow f^*\omega_B \rightarrow \Omega_\cC^1 \rightarrow \Omega_{\cC/B}^1 \rightarrow 0$, and the {\em kernel bundle} is $\mathcal{K} = \ker\theta$. In the non-smooth case, the extensions of these over the singular locus of $f$ define the same bundles as above.

By construction there are inclusions $\cU\subseteq\cK\subseteq f_*\omega_f$, which combined with the splitting $f_\ast\omega_f\cong\cU\oplus\cA$ give an exact sequence 
$$0\to \cK/\cU\to\cA\to f_*\omega_f/\cK\to 0,$$
exhibiting $\cK/\cU$ as a vector subbundle of $\cA$. If $\cU\neq\cK$, $\cK/\cU$ has negative curvature by \cite{Z00}, hence $f_*\omega_f/\cK$ has bigger degree than $\cA$.

Our main tool in order to understand how $\cK$ can be larger than $\cU$  is given by the following 

\begin{lemma} \label{lem-relrigRanks}
Let $f$ be minimally supported on a divisor $\cD$ with $\cD \cdot C_b = d$ and $h^0(C_b,\mathcal{O}_{C_b}(\cD_{|C_b}))=1$ for general $b \in B$ (i.e. $\cD$ is relatively rigid). Then $1.$  $\rk \mathcal{K}=g-d$ and $2.$ $\rk \mathcal{U} \leq \frac{g+1}{2}$.
\end{lemma}
The proof follows the line of \cite{GST17}.
\begin{proof}
$1.$ It follows from Lemma \ref{Lem-BoundKer}. For a general $b \in B$, we indeed have
$$\cK \otimes \bC\left(b\right) = \ker\left(\cup\xi_b: H^0\left(C_b,\omega_{C_b}\right) \longrightarrow H^1\left(C_b,\cO_{C_b}\right)\right).$$ Since $h^0(C_b,\mathcal{O}_{C_b}(\cD_{|C_b}))=1,$ then 
$$\rk \cK = \dim \cK \otimes \bC(b) = \dim \ker (\cup\xi_b) = g - \deg(\cD_{|C_b}) = g - d.$$

$2.$ The argument follows the line of \cite[section 3.1, case $1$]{GST17} (see also \cite[Lemma 3.2]{PT}). Assume that $\rank \cU\geq 2$, otherwise there is nothing to prove. By assumption $f$ is supported on a relatively rigid divisor, meaning that the divisor restricts to a rigid divisor on any smooth fibre. Then, we can lift a basis $\eta_1,\ldots,\eta_{u_f}$ of flat sections of $\cU \subseteq \cK$, to a set $\omega_1,\ldots,\omega_{u_f} \in H^0\left(\cC,\Omega_\cC^1\right)$ of linearly independent {\em closed} 1-forms with the property that any two of these forms wedge to zero. Applying the ``Tubular Castelnuovo-de Franchis'' (see \cite[Theorem 1.4]{GST17}), we get a family $\varphi_b:C_b\to C$ of morphisms from the general fibre $C_b$ of $f$ to a fixed curve $C$ of genus $g(C) \geq u_f=\rk \cU$. By the Riemann-Hurwitz formula, 
$$ 2g-2=\deg \varphi_b (2g(C)-2)+\deg R_b \geq 2 \deg \varphi_b (u_f-1),$$
where $R_b$ is the ramification divisor of $\varphi_b$.  In particular,
$$ g-1\geq\deg \varphi_b (u_f-1),$$
and so for $u_f>\frac{g+1}{2}$ and $g\geq 2$, one has $\deg \varphi_b=1$ and hence a isotrivial family.
\end{proof}

\begin{lemma}\label{lem:relrigmonodromy} Let $f$ be minimally supported on a relatively rigid divisor. Then $\cU$ has finite monodromy. 
\end{lemma}
\begin{proof} We can assume $\rank\cU\geq2$ (in the case of rank $1$, the monodromy is finite since the line bundle must be torsion, proven e.g. in \cite{barja-fujita}). Repeating the argument given in $2.$ of the proof of Lemma \ref{lem-relrigRanks}, we have that our bundle satisfies the assumptions of \cite[Theorem 0.2]{PT} and thus has finite monodromy.
\end{proof}

We end this section by providing a way to construct non-isotrivial local families of curves with $\mathcal{K}$ of any given rank between $1$ and $g-1$.
\begin{theorem}\label{thm:localfam}
Let $C$ be any curve of genus $g \geq 3$. Then for any $0 \leq r \leq g-1$ there are one-dimensional deformations of $C$ with $\rk \mathcal{K} = r$.
\end{theorem}
\begin{proof}
Let us first consider a more geometric interpretation of supporting divisors. Let $C$ be a curve of genus $g$ and $\phi: C \rightarrow \mathbb{P}(H^0(C,\omega_C^{\otimes2})^{\vee}) \cong \mathbb{P}^{3g-4}$ its bicanonical embedding. Given an effective divisor $D$ in $C$, we define its {\em span} as
$$\left\langle D \right\rangle := \cap_{D \leq \phi^*H} H = \mathbb{P}\left(H^0(C,\omega_C^{\otimes2}(-D))^{\perp}\right),$$
i.e., the intersection of all hyperplanes cutting out a divisor on $C$ that contains $D$, which coincides with the projectivization of the annihilator of $H^0(C,\omega_C^{\otimes2}(-D))$. In particular, if $\deg D < \deg \omega_C = 2g-2$ then Riemann-Roch gives $\dim\left\langle D \right\rangle = \deg D - 1$.

Let now $\xi \in H^1(C,T_C)$ be a non-zero first-order infinitesimal deformation, which defines a point $[\xi] \in \mathbb{P}(H^1(C,T_C)) \cong \mathbb{P}(H^0(C,\omega_C^{\otimes2})^{\vee})$. It is just a reformulation of the definitions that a divisor $D$ supports $\xi$ if and only if $[\xi] \in \left\langle D \right\rangle$. Thus first-order deformations supported on a divisor $D$ correspond to points in $\left\langle D \right\rangle$. Furthermore, $\left\langle D' \right\rangle \subsetneq \left\langle D \right\rangle$ for any $0 \leq D' < D$ if and only if $\omega_C^{\otimes2}(-D)$ has no base points, e.g. if $\deg D \leq 2g-4$. In this case, the first-order deformations {\em minimally} supported in $D$ form a non-empty Zariski-open subset $\left\langle D \right\rangle^{\circ}$ of $\left\langle D \right\rangle$, namely the complement of the spans of the finitely many strict subdivisors of $D$.

We want to focus on the deformations supported on rigid divisors of a given degree $d \in \{1,\ldots,g\}$. For any such $d$, the map $C^{(d)}=Div^d(C) \rightarrow Pic^d(C)$ is generically one-to-one, thus the rigid divisors form a non-empty Zariski-open set $V_d \subseteq C^{(d)}$. Let
\begin{multline*}
X_d = \left\{(D,[\xi]) \,\mid\, \deg D=d, \, [\xi] \in \left\langle D \right\rangle\right\} = \bigcup_{D \in C^{(d)}}\{D\}\times\left\langle D \right\rangle \subset C^{(d)} \times \mathbb{P}(H^0(C,\omega_C^{\otimes2})^{\vee})
\end{multline*}
be the obvious incidence variety, which is irreducible of dimension $2d-1$ because $\dim \left\langle D \right\rangle = d-1$ for $d < 2g-2$. The subset
$$X_d^{\circ} = \bigcup_{D \in V_d}\{D\}\times\left\langle D \right\rangle^{\circ} \subset X_d$$
is a dense open subset. Indeed, its complement $X_d\setminus X_d^{\circ}$ is contained in the union of
\begin{enumerate}
\item the Zariski-closed strict subset $\left(V_d\setminus V_{d-1}\right)\times \mathbb{P}(H^0(C,\omega_C^{\otimes2})^{\vee})$, and
\item the image of $X_{d-1} \times C \rightarrow X_d$, defined by $(D',[\xi],p) \mapsto (D'+p,[\xi])$, which has dimension at most
$$\dim X_{d-1} + \dim C = 2(d-1)-1+1 = 2d-2 < \dim X_d.$$
\end{enumerate}
Set also $Y_d = p_2(X_d) \subset \mathbb{P}(H^0(C,\omega_C^{\otimes2})^{\vee})$, which by the above discussion corresponds to the (closed) set of infinitesimal deformations supported on some divisor of degree $d$. Of course, $Y_d$ coincides with the $d$-th secant variety of $C \subset \mathbb{P}(H^0(C,\omega_C^{\otimes2})^{\vee})$. Define also the dense subset
$$Y_d^{\circ} = p_2(X_d^{\circ}) \subset Y_d,$$
which corresponds to the first-order deformations {\em minimally} supported on some divisor of degree $d$. Thus, for any $[\xi] \in Y_d^{\circ}$, there is some minimal supporting divisor $D$ of degree $d$ and $r(D)=0$, and hence by Lemma \ref{Lem-BoundKer}
$$\dim\ker(\cup\xi: H^0(C,\omega_C) \rightarrow H^1(C,\mathcal{O}_C)) = g - d.$$

Let now $\pi:\mathcal{C} \rightarrow \Delta$ be a semiuniversal deformation of $C$ over some $(3g-3)$-dimensional polydisk $\Delta$, $\mathbb{P} = \mathbb{P}_{\Delta}(\pi_*\omega_{\pi}^{\otimes2}) \rightarrow \Delta$ and $\phi: \mathcal{C} \rightarrow \mathbb{P}$ the relative bicanonical map.
We can mimick the above construction on every fibre of $\pi$ and obtain a non-empty locally closed subset $\mathcal{Y}_d^{\circ} \subset \mathbb{P}$ that surjects onto $\Delta$. Indeed, if $\mathcal{C}_{\Delta}^{\left(d\right)}=Div^d\left(\mathcal{C}/\Delta\right)$ denotes the relative symmetric $d$-th product of $\mathcal{C}$, we can consider the Zariski-open subset $\mathcal{V}_d\subseteq\mathcal{C}_{\Delta}^{\left(d\right)}$ corresponding to rigid divisors and the incidence subvariety
$$\mathcal{X}_d=\left\{\left(D,\xi,t\right)\,\mid\,t\in\Delta,D\in Div^d\left(C_t\right),\xi\in H^1\left(C_t,T_{C_t}\right),\left[\xi\right]\in\left\langle D\right\rangle\right\}\subseteq\mathcal{C}_{\Delta}^{\left(d\right)}\times_{\Delta}\mathbb{P}.$$
The announced $\mathcal{Y}_d^{\circ}$ is then the image by the projection to $\mathbb{P}$ of the (non-empty) open subset
$$\mathcal{X}_d^{o}=\left(\mathcal{X}_d\setminus\left(\mathcal{X}_{d-1}\times_{\Delta}\mathcal{C}\right)\right)\cap\left(\mathcal{V}_d\times_{\Delta}\mathbb{P}\right).$$
Up to shrinking $\Delta$, we can find a section $\sigma: \Delta \rightarrow \mathcal{Y}_d^{\circ}$, which thus at every point $b \in B$ defines (up to scalar) a first-order deformation $\xi_b$ minimally supported on a divisor of degree $d$.

Since $T_b\Delta \cong H^1(C_b,T_{C_b})$ for any $b \in \Delta$, the relative bicanonical space $\mathbb{P}$ can be identified with the projectivization of the tangent bundle of $\Delta$. In this way, any section $\sigma: \Delta \rightarrow \mathcal{Y}_d^{\circ}$ can be thought of as a rank-one (hence automatically integrable) distribution on $\Delta$. If $B \subset \Delta$ is any integral curve of a given $\sigma$, the restriction $f=\pi_B: \pi^{-1}(B) \rightarrow B$ gives the desired family.
\end{proof}

These families are constructed over a disk. One could thus wonder, if such examples can exist over a quasi-projective base $B$. The answer is yes, as our main results and also some more explicit examples contructed in section \ref{sec:cyclicexamples} show.

\section{Semistable families of cyclic coverings of $\bP^1$ with $\cK$ larger than $\cU$.}\label{sec:cyclicexamples}
In this section we construct semistable families of curves over a projective base with $\cU \subsetneq \cK$ by moving few branch points of a low degree covering. The largest range for $\rk\cK$ is achieved by families of hyperelliptic curves. Our main tool is the following

\begin{proposition}\label{prop:proj-examples}
Let $\pi: C \rightarrow \mathbb{P}^1$ be a simple cyclic covering of degree $n$ with reduced branch divisor $B=q_1+\ldots+q_m$ ($n\mid m$) and suppose $g\left(C\right)=g\geq2$. Let $f: \cC \rightarrow \Delta$ be a deformation of $C$ obtained by moving the branch points $q_1,\ldots,q_k$. If $k < \frac{m}{n}$, then
$$\rk \cK = g-(n-1)k = \frac{(n-1)(m-2-2k)}{2} \quad \text{and} \quad \rk \cU \leq \frac{g+1}{2}.$$
In particular, if $k < \frac{g-1}{2(n-1)} = \frac{mn-2n-m}{4(n-1)}$, then $\cU \subsetneq \cK$.
\end{proposition}
\begin{proof}
For each $i=1,\ldots,k$ let $p_i=\pi^{-1}\left(q_i\right)$ be the ramification point above $q_i$, and set $D=\left(n-1\right)\left(p_1+\ldots+p_k\right)$, the {\em variable ramification divisor}. We will show that $D$ is a rigid divisor that supports $f$ minimally, and Lemma \ref{lem-relrigRanks} gives the final assertions.

In order to show that $D$ is a rigid divisor, let us consider first the divisor $D'=\frac{n}{n-1}D=n(p_1+\ldots+p_k) = \pi^*(q_1+\ldots+q_k)$. It holds then
\begin{multline*}
H^0\left(C,\cO_C\left(D'\right)\right) = H^0\left(C,\pi^*\cO_{\bP^1}\left(q_1+\ldots+q_k\right)\right) \stackrel{\pi^*}{\cong} H^0\left(\bP^1,\cO_{\bP^1}\left(q_1+\ldots+q_k\right)\otimes\left(\pi_*\cO_C\right)\right) \cong \\
\cong \bigoplus_{i=0}^{n-1} H^0\left(\bP^1,\cO_{\bP^1}\left(q_1+\ldots+q_k\right)\otimes\cO_{\bP^1}\left(-i\frac{m}{n}\right)\right) = H^0\left(\bP^1,\cO_{\bP^1}\left(q_1+\ldots+q_k\right)\right),
\end{multline*}
where the last equality follows from the hypothesis $k < \frac{m}{n}$, hence
$$\deg\left(\cO_{\bP^1}\left(q_1+\ldots+q_k\right) \otimes \cO_{\bP^1}\left(-i\frac{m}{n}\right)\right)<0$$
for any $i>0$.

This shows that any meromorphic function in $H^0(C,\cO_C(D'))$ is the pull-back of a meromorphic function in $H^0(\bP^1,\cO_{\bP^1}(q_1+\ldots+q_k))$. In particular, any non-constant function in $H^0(C,\cO_C(D'))$ has poles of order exactly $n$ at some $p_i$, and hence $H^0(C,\cO_C(D)) \subseteq H^0(C,\cO_C(D'))$ consists only of the constant functions, i.e. $D$ is rigid.

It remains to show that $D$ is a minimal supporting divisor of $f$. The genus of $C$ is $g=\frac{(m-2)(n-1)}{2}$, and thus $\deg\left(T_C\left(D\right)\right)\leq0$ with equality if and only if $n=k=g=2$ (hence $m=6$). In this last case an argument along the lines above shows that $H^0\left(C,T_C\left(D\right)\right)=0$. Lemma \ref{lem:Hori-divisor} can thus be applied in any case, giving that $D$ is a minimal rigid supporting divisor.
\end{proof}

We now construct a family of deformations as in Proposition \ref{prop:proj-examples} over a projective base as follows. Fix two points $0,\infty \in \mathbb{P}^1$ and integers $2 \leq n < m$ such that $n \mid m$, and $1 \leq k < \frac{m}{n}$. For each $i=1,\ldots,k$ let $L_i' \subset \mathbb{P}^1 \times \mathbb{P}^1$ be a curve of bidegree $(1,1)$ with $L_i' \cap \{0\}\times\mathbb{P}^1 = \{(0,q_i)\}$, or equivalently, the graph of an automorphism $\phi_i: \mathbb{P}^1 \rightarrow \mathbb{P}^1$ with $\phi_i(0)=q_i$. For $i=k+1,\ldots,m$ let $L_i' = \mathbb{P}^1 \times \{q_i\} \subset \mathbb{P}^1 \times \mathbb{P}^1$. For a general choice of the lines $L_i'$ we can assume that all of them intersect transversely in $k(m-k)+2\binom{k}{2} = k(m-1)$ different points $t_1,\ldots,t_{k(m-1)}$, none of them lying on $M'=\{\infty\}\times\mathbb{P}^1$. In this case the divisor $M'+\sum_{i=1}^m L_i'$ has simple normal crossings. For each $i=1,\ldots,k(m-1)$ let $s_i = \pi_1(t_i)$, where $\pi_1\colon\mathbb{P}^1 \times \mathbb{P}^1 \rightarrow \bP^1$ denotes the projection onto the first factor.

Our idea is to define a degree $n$ cyclic covering of $\mathbb{P}^1 \times \mathbb{P}^1$ branched along the lines $L_i'$, so that the family defined by the projection onto the first $\bP^1$ looks like the deformations in Proposition \ref{prop:proj-examples} around the smooth fibres. By degree considerations we need to introduce branch also over $rM'$, where $r \in \mathbb{Z}_{\geq0}$ is such that $k+r$ is a multiple of $n$. However, the minimal desingularization of such a cyclic covering does not lead to a semistable fibration over $\mathbb{P}^1$. One can indeed check this explicitly from the equations of such a covering, noticing that some non-reduced components appear over the singularities. 

In order to solve this problem, let $\sigma: B \rightarrow \mathbb{P}^1$ be the (normalization of the) cyclic covering of degree $n$ branched over $s_1+\ldots+s_{k(m-1)}+k\infty$, set $\widetilde{\sigma} = \sigma \times \id: B \times \mathbb{P}^1 \rightarrow \mathbb{P}^1\times \mathbb{P}^1$ and define $L_i = \widetilde{\sigma}^* L_i' \subseteq B \times \mathbb{P}^1$. The new divisor $D=\sum_{i=1}^m L_i$ has no longer simple normal crossings, but each singular point has local equations of the form $y(y-x^n)=0$. We can anyway construct the minimal desingularization $Z \rightarrow B \times \mathbb{P}^1$ of the cyclic covering of degree $n$ ramified over $D \subset B \times \bP^1$, and define $f: Z \rightarrow B$ as the induced fibration. Note that in this case no extra ramification is needed. Indeed, $D$ is a section of the line bundle
$$\cL=\widetilde{\sigma}^*\cO_{\bP^1\times\bP^1}(k,m) = \sigma^*\cO_{\bP^1}(k) \boxtimes \cO_{\bP^1}(m),$$
which has an $n$-th root because its bidegree is $(nk,m)$, a multiple of $n$.

We have thus constructed families over a projective base with ranks described by  proposition \ref{prop:proj-examples} as claimed. By Lemma \ref{lem:relrigmonodromy}, we can furthermore conclude that the monodromy of $\cU$ is always finite in these cases.

Summing up, we have the following
\begin{theorem}\label{thm:cyclicexamples}
The fibraton $f: Z \rightarrow B$ constructed above is semistable. The general fibre is a degree $n$ cyclic covering of $\mathbb{P}^1$ ramified over $m$ points (of which $k$ vary with $b \in B$). The genus of a general fibre is $g=\frac{(n-1)(m-2)}{2}$. There are $k(m-1)$ singular fibres, which consist of the transverse union of a curve of genus $g-1$ and an elliptic curve. They have $\\rk \cK = g-(n-1)k = \frac{(n-1)(m-2-2k)}{2}$ and $\cU$ has finite monodromy group and rank at most $\frac{g+1}{2}$.
\end{theorem}

Our constructions are inspired by a series of examples studied by Caanese and Dettweiler in \cite{CatDet_Vector_2016}, where also degree $n$ cyclic coverings are considered, but ramified only over 4 points (with multiplicities). Although they do a more general analysis, we will here focus on what they call ``standard case'', which have $\gcd(n,6)=1$, three of the branch points have multiplicity one and the fourth one has multiplicity $n-3$. Moving one of the ramification points defines a family over $\bP^1$, which becomes semistable after a degree $n$ covering of the base (and a desingularization). Let $f: S \rightarrow B$ be the resulting fibration. The genus of the smooth fibres is $g=n-1$ and the singular fibres consist of two curves of genus $\frac{n-1}{2}$ meeting transversely in one point. It holds $q(S)=g(B)=\frac{n-1}{2}$, hence $f$ is the Albanese map of $S$. More details can be found in \cite[Section 4]{CatDet_Vector_2016}. These families provide examples where $\cU$ has non finite monodromy group, so they behave very different from ours, where we have seen that the monodromy is finite.

The rank of their flat unitary summand $\cU$ has been studied by Lu in \cite{Lu18}, where arbitrary $n\geq 4$ is are also consiedered, proving the lower bounds
\begin{equation}\label{eq:rk-U-Lu}
\rk\cU\geq
\begin{cases}
\left\lfloor\frac{2g+1}{3}\right\rfloor & \text{ if $n \equiv 1 \mod 3$} \\
\left\lfloor\frac{2g-2}{3}\right\rfloor & \text{ if $n \not\equiv 1 \mod 3$.}
\end{cases}
\end{equation}
With our techniques we are able to prove furthermore the following. 

\begin{proposition}
Let $f:S \rightarrow B$ be as in the ``standard cases'' of \cite{CatDet_Vector_2016}. Then $\mathcal{U}=\mathcal{K}$ and equality holds in \eqref{eq:rk-U-Lu}.
\end{proposition}
\begin{proof}
Since $\cU\subseteq\cK$, we only have to prove $\rk\mathcal{K}\leq\left\lfloor\frac{2g+1}{3}\right\rfloor$ if $n \equiv 1 \mod 3$, and $\rk\mathcal{K}\leq\left\lfloor\frac{2g-2}{3}\right\rfloor$ otherwise.

By construction, $f$ is a family obtained by moving one branch point of a morphism from a general fibre to $\bP^1$. So we can apply Lemma \ref{lem:Hori-divisor} to obtain a minimally supporting divisor $\cD=(n-1)\cP \subseteq S$, where $\cP$ is the section defined by moving the branch point. By Lemma \ref{Lem-BoundKer} we have that
$$\rk\mathcal{K}\leq g-\deg D+2r(D) = 2r(D),$$
where $D=\cD_{|C}=(n-1)P$ is the restriction to a general fibre $C$ of $f$. In order to compute $r(D)= r((n-1)P)$ note first that $r((n-1)P)= r(nP)-1$ because $P$ is not a base point of $|nP|$. Indeed, since $nP=\pi^*Q$ for some $Q \in \bP^1$, the pull-back of any other $Q' \in \bP^1$ is a divisor linearly equivalent to $nP$ not contaning $P$. Secondly, since $C \rightarrow \bP^1$ is a morphism of degree $n$ ramified over a divisor of the form $R=(n-3)0+1+Q+\infty$, i.e. also of degree $n$, we can apply \cite[Corollary 3.11]{EsnVie} with $\cL=\cO_{\bP^1}(1)$ and obtain
$$\pi_*\cO_C = \bigoplus_{i=0}^{n-1}\cO_{\bP^1}\left(-i+\left\lfloor\frac{i}{n}R\right\rfloor\right) = \bigoplus_{i=0}^{n-1}\cO_{\bP^1}\left(-i+\left\lfloor\frac{i(n-3)}{n}\right\rfloor\right) = \bigoplus_{i=0}^{n-1}\cO_{\bP^1}\left(\left\lfloor\frac{-3i}{n}\right\rfloor\right).$$
This implies
\begin{multline*}
h^0(\cO_{C}(nP))=h^0(\pi^*\cO_{\bP^1}(1))=h^0(\pi_\ast\pi^\ast\cO_{\bP^1}(1))=h^0(\cO_{\bP^1}(1)\otimes \pi_\ast \cO_{C}) = \\
= \sum_{i=0}^{n-1} h^0\left(\cO_{\bP^1}\left(1+\left\lfloor\frac{-3i}{n}\right\rfloor\right)\right) = \left\lfloor\frac{n}{3}\right\rfloor+2,
\end{multline*}
and thus $r(D) = r(nP)-1 = \left\lfloor\frac{n}{3}\right\rfloor$. Explicitly writing $\left\lfloor\frac{n}{3}\right\rfloor$ in each case leads to the desired upper-bound for $\rk \cK$.
\end{proof}

\section{Proof of the main theorems}\label{sec:pf:main}

In this section we give the proof of Theorem \ref{thm:main} and Corollary \ref{cor:main}. 

\begin{proof}[Proof of  Theorem \ref{thm:main}]
We say that a ramified covering $\pi\colon C\to\bP^1$ is {\em simply ramified} at $p\in\bP^1$ if $p$ is the only ramification point on its fibre and its ramification index is 2.

We first show that for any subset $\left\{p_1,\ldots,p_g\right\}$ of $g$ distinct points of $C$ there is a covering $\pi\colon C\to\bP^1$ simply ramified at each $p_1,\ldots,p_g$. To this aim, we fix an embedding $C\hookrightarrow\bP^n=\bP\left(H^0\left(C,L\right)\right)$ given by a complete linear system $|L|$ of degree $d\geq5g+3$ and consider morphisms $\pi_H\colon C\to\bP^1$ given by projection from a linear subspace $H\subset\bP^n$ of codimension $2$ and disjoint from $C$.

The condition on the degree assures that $h^0\left(C,L\left(-D\right)\right)=h^0\left(C,L\right)-\deg D$ for any effective divisor $D$ with $\deg D\leq3g+2$. In particular, for any $p\in C$ the tangent line $L_p=T_pC$ and the osculating plane $\Pi_p$ are given by
$$L_p=\bP\left(H^0\left(C,L\left(-2p\right)\right)^{\perp}\right)\cong\bP^1\qquad\text{and}\qquad\Pi_p=\bP\left(H^0\left(C,L\left(-3p\right)\right)^{\perp}\right)\cong\bP^2,$$
where $\perp$ denotes the annihilator inside $H^0\left(C,L\right)^{\vee}$. Moreover, for any distinct $p_1,\ldots,p_g,p\in C$ the osculating planes $\Pi_{p_1},\ldots,\Pi_{p_g}$ and the tangent line $L_p$ are independent, in the sense that the linear span
$$\left\langle\Pi_{p_1},\ldots,\Pi_{p_g},L_p\right\rangle\subset\bP^n$$
has dimension $2g+1$, the maximal possible.

Consider now a linear subspace $H\subset\bP^n$ of codimension 2 and disjoint from $C$. Then $\pi_H$ is ramified at $p$ if and only if $L_p\subseteq\left\langle p,H\right\rangle$, (i.e. if $L_p\cap H\neq\emptyset$) and the ramification index is exactly $2$ if and only if $\Pi_p\not\subset\left\langle p,H\right\rangle$ (i.e. if $\Pi_p\cap H=L_p\cap H$). On the other hand, it holds $\pi_H\left(p\right)=\pi_H\left(q\right)$ for $p\neq q$ if and only if $\left\langle p,H\right\rangle=\left\langle q,H\right\rangle$, or equivalently $H$ intersects the line $\overline{pq}$.

Let now $p_1,\ldots,p_g\in C$ be arbitrary distinct points and for each $i=1,\ldots,g$ pick $q_i\in L_{p_i}, q_i\neq p_i$. It is now easy to show that the set of codimension-$2$ linear subspaces $H$ containing $q_1,\ldots,q_g$ and such that $\pi_H$ is ramified at each $p_i$ with index $2$ and $\pi_H\left(p_i\right)\neq\pi_H\left(p_j\right)$ for $i\neq j$ form a Zariski-open subset of the Grassmannian $\bG$ of codimension-$2$ subspaces containing the $q_1,\ldots,q_g$. It remains to achieve the simple ramification at each $p_1,\ldots,p_g$, i.e. no other ramification point has the same image as any $p_i$. By the above discussion, the covering $\pi_H$ is ramified at another given point $p\in C$ if and only $L_p\cap H\neq\emptyset$, which is a codimension-$2$ condition on $\bG$ (because of the condition $\deg L\geq 5g+3$). By moving $p\in C$ we see that the set of ``bad'' subspaces $H$ (such that $\pi_H$ is not simply ramified at $p_1,\ldots,p_g$) has codimension at least $1$ in $\bG$, hence a general $H\in\bG$ defines a covering simply ramified at $p_1,\ldots,p_g$, as wanted.
 
Suppose now in addition that the points $p_1,\ldots,p_g$ form a rigid divisor on $C$ (which happens for a set of $g$ points in general position on $C$) and pick a covering $\pi\colon C\to\bP^1$ as above, simply ramified at $p_1,\ldots,p_g$. Denote by $b_1=\pi\left(p_1\right),\ldots,b_g=\pi\left(p_g\right),b_{g+1},\ldots,b_k$ the branch points of $\pi$. To finish the proof we construct a one-dimensional family $f\colon\mathcal{C}\rightarrow B$ of deformations of $C$ over a quasi-projective base $B$, moving $r$ of the branch points $b_1,\ldots,b_g$ as follows.

For $i=1,\ldots,r$ let $\Delta_i\subset\bP^1$ be a disk centred in $b_i$, small enough so that $\Delta_i\cap\Delta_j=\emptyset$ for $i\neq j$, and also $b_j\not\in\Delta_i$ for $i=1,\ldots,r$ and $j=r+1,\ldots,k$. By the Riemann-existence theorem, for any $t=(t_1,\ldots,t_r)\in \prod_{i=1}^r\Delta_i=\Delta$ there is a covering $\pi_t:C_t\to \bP^1$ branched on $\{t_1,t_2,\dots,t_r,q_{r+1},\dots,q_k\}$ with the same ramification data as $\pi$. These coverings vary holomorphically over the polydisk $\Delta$, and thus define an $r$-dimensional family $f:\cC\to \Delta$ with $f^{-1}(t)=C_t$. Because of monodromy reasons, this family might not extend automatically over the quasi-projective variety $X=(\bP^1)^r\setminus Z$, where
$$Z=\{(t_1,\ldots,t_r)\in(\bP^1)^r\,\mid\,t_i=t_j\,\text{ or }\,t_i=b_j \,\text{ for some } i\neq j\}$$
is the set where two branch points collide. We can anyway extend it over any simply connected open set containing $\Delta$ and so in particular over the universal covering $\psi:\tilde{X}\to X$, which is however not quasi-projective.

Nevertheless, for given $t\in X$ there are only finitely many coverings (up to isomorphism) branched over $\{t_1,t_2,\dots,t_r,b_{r+1},\dots,b_k\}$, and the fundamental group $\pi_1\left(X\right)$ acts naturally on this finite set. The kernel $G$ of the induced group homomorphism $\rho:\pi_1(X)\to \Sigma_N$ into the symmetric group $\Sigma_N$ (for some appropriate $N$) has therefore finite index in $\pi_1\left(X\right)$ and is independent of $t\in X$ general. The family over $\tilde{X}$ induces thus a family $f\colon\cC\to Y$ over $Y=\tilde{X}/G$, which is a finite covering of the quasi-projective variety $X$, hence quasi-projective itself.

To finish the proof we consider a quasi-projective curve $B\subset Y$ through a point $t_0$ of $Y$ above $(b_1,\ldots,b_r)\in X$ corresponding to $\pi$, and transverse to the ``coordinate hypersurfaces'' $\{t_i=b_i\}$. Possibly after a finite base change this family can be extended to a semistable one over a projective base. The fact that $\rk\cK=g-r$ follows directly from Lemmas \ref{lem:Hori-divisor} and \ref{Lem-BoundKer}. Indeed, Lemma \ref{lem:Hori-divisor} shows that for $t_0\in B$ the infinitesimal deformation is minimally supported on $D=p_1+\ldots+p_r$, hence in particular is not isotrivial (see Remark \ref{rmk:non-isotr}). By construction of $\pi$, the divisor $D$ is rigid and Lemma \ref{lem-relrigRanks} gives both $\rk\cK=g-r$ and $\rk\cU\leq\frac{g+1}{2}$.
\end{proof}

\begin{proof}[Proof of Corollary \ref{cor:main}] The proof is a straightforward application of Theorem \ref{thm:main} together with the following argument about the monodromy of the flat summand. Since $C$ is a smooth curve with simple Jacobian variety $J\left(C\right)$, the flat bundle $\cU$ of any one-dimensional family $f:\cC\to B$ through $C$ must  be either zero or have infinite monodromy. Otherwise, $\cU$ would become trivial after a finite \'etale base change, defining an abelian subvariety of $J\left(C\right)$ and contradicting its simplicity. However the family $f:\cC\to B$ as contructed in the proof of Theorem \ref{thm:main} is minimally supported on a relatively rigid divisor. Lemma \ref{lem:relrigmonodromy} implies that $\cU$ has finite monodromy, hence it must be zero by the above discussion.
\end{proof} 

\bibliographystyle{alpha}


\end{document}